\documentclass[]{amsart}

\usepackage{amsmath,amsthm,amssymb}
\usepackage{mathrsfs}
\usepackage[all]{xy}
\usepackage{tikz}
\usepackage{diagbox}
\usepackage{hyperref}

\newtheorem{theorem}{Theorem}[section]

\newtheorem{lemma}[theorem]{Lemma}
\newtheorem{corollary}[theorem]{Corollary}

\theoremstyle{definition}
\newtheorem{definition}[theorem]{Definition}

\newtheorem{remark}[theorem]{Remark}

\newcommand{\holim}{\operatornamewithlimits{holim}}

\title{Condensation of the operad for multiplicative hyperoperads}

\author[Florian De Leger]{Florian De Leger}
\address{Mathematical Institute of the Academy \\ \v{Z}itn\'a 25, 115~67 Prague 1, Czech Republic}
\email{de-leger@math.cas.cz}

\author{Maro\v{s} Grego}
\address{Faculty of Mathematics and Physics, Charles University \\
Sokolovsk\'a 49/83, 186~75 Prague 8, Czech Republic}
\email{maros@grego.site}


\thanks{The first author was supported by RVO:67985840 and Praemium Academiae of Martin Markl}
	
\begin{document}
	
\begin{abstract}
	We construct a map of operads from an $E_2$-operad to the condensation of the operad for multiplicative hyperoperads. We deduce from it the existence of an $E_2$-action on the homotopy limit of the underlying functor of a multiplicative hyperoperad. This result is the higher dimensional analogue of a result due to Batanin and Berger implying Deligne's conjecture.
\end{abstract}
	
\maketitle

\tableofcontents

\section{Introduction}

In order to define weak $n$-categories, Baez and Dolan \cite{baezdolan} introduced a notion of \emph{plus construction}. This construction creates a new (symmetric coloured) operad $P^+$ from an operad $P$. In \cite{delegergrego}, we studied the geometric aspects of the iterations of this construction. More specifically, we focused on the operad $I^{+++}$ obtained by iterating this construction three times, starting from the initial one-coloured operad $I$. We called \emph{hyperoperads} the algebras for this operad. The explicit definition of a (multiplicative) hyperoperad will be given in Subsection \ref{subsectionmultiplicativehyperoperads} and the description of the operad $\mathcal{H}$ for multiplicative hyperoperads will be given in Subsection \ref{subsectionoperadformulthyperop}. We proved \cite[Corollary 6.13]{delegergrego} that for a multiplicative hyperoperad $\mathcal{A}$ satisfying some reduceness conditions, there is a triple delooping
\[
	\Omega^3 \mathrm{Map}_{\mathrm{HOp}} (\zeta,u^*(\mathcal{A})) \sim \holim_{\Omega_p} \mathcal{A}^\bullet,
\]
where $\Omega$ is the loop space functor, $\mathrm{Map}_{\mathrm{HOp}}(-,-)$ is the homotopy mapping space in the category of hyperoperads, $\zeta$ is the terminal hyperoperad, $u^*$ is the obvious forgetful functor and $\mathcal{A}^\bullet$ is the underlying functor of $\mathcal{A}$ (see Subsection \ref{subsectioncolouredoperads}). This result is the higher dimensional analogue of the double delooping proved independently by Turchin \cite{turchin} and Dwyer and Hess \cite{dwyerhess}.

For $m \geq 0$, $m$-fold loop spaces are related to $E_m$-operads through May's recognition principle \cite{may}. Therefore, a natural question to ask is whether there is in general an $E_3$-action on the homotopy limit of the underlying functor of a multiplicative hyperoperad. This is what we investigate in this paper. We will follow the techniques developed by Batanin and Berger in \cite{BataninBergerLattice}. In this paper, the authors introduced the notion of \emph{condensation} of a (symmetric) coloured operad to a one-coloured operad. They also introduced a coloured operad $\mathcal{L}$ called \emph{lattice path operad} together with a filtration of it. They computed the condensation of the $m$-th stage filtration $\mathcal{L}_m$ and deduced from it the existence of an $E_m$-action on the totalization of the underlying cosimplicial object of an algebra over $\mathcal{L}_m$. Similarly, we would like to compute the condensation of the operad for multiplicative hyperoperads. Our first conjecture was that the resulting one-coloured operad is an $E_3$-operad. However, this does not seem to be the case. Still, we are able to construct a map from an $E_2$-operad to this resulting one-coloured operad.

Here is a more detailed description of the contents of this paper. In Section \ref{sectionpreliminaries}, we will recall the notions of a coloured operad and its condensation. We will also recall the definition of the complete graph operad and multiplicative hyperoperads. In Section \ref{sectioncondensation}, we will prove our main result which is Theorem \ref{theoremmapfrome2tocoend}. We will start by giving the explicit description of the operad $\mathcal{H}$ for multiplicative hyperoperads. As in \cite{BataninBergerLattice}, the crucial ingredient is the construction of a complexity map $c$. In our case, it is a map from $\mathcal{H}$ to the third stage filtration $\mathcal{K}_3$ of the complete graph operad. This map $c$ allows us to define another coloured operad $\hat{\mathcal{H}}$, and the main argument in the proof of Theorem \ref{theoremmapfrome2tocoend} is that the condensation of $\hat{\mathcal{H}}$ yields an $E_2$-operad. We will also rely on Lemma \ref{lemmaformulacoend}, which says that the condensation can be computed using classifying spaces of categories. This allows us to use the tools of homotopy theory, such as Quillen's Theorem A and Thomason's Theorem.

\subsection*{Acknowledgement}

We want to thank Clemens Berger for our interesting discussions on this topic during the workshop \emph{Homotopical Algebra and Higher Structures} in Oberwolfach.

\section{Preliminaries}\label{sectionpreliminaries}

%
%

\subsection{Coloured operads and their algebras}\label{subsectioncolouredoperads}

Recall that a (symmetric) \emph{coloured operad} $\mathcal{O}$ in a symmetric monoidal category $(\mathcal{E},\otimes,e)$ is given by a set of \emph{colours} $I$ and for each integer $k \geq 0$ and $k+1$-tuple of colours $(i_1,\ldots,i_k;i)$, an object $\mathcal{O}(i_1,\ldots,i_k;i) \in \mathcal{E}$, together with
\begin{itemize}
	\item for all $i \in I$, a map $e \to \mathcal{O}(i;i)$ called \emph{unit},
	\item \emph{multiplication} maps
	\[
	m: \mathcal{O}(i_1,\ldots,i_k;i) \otimes \mathcal{O}(i_{11},\ldots,i_{1l_1};i_1) \otimes \ldots \otimes \mathcal{O}(i_{k1},\ldots,i_{kl_k};i_k) \to \mathcal{O}(i_{11},\ldots,i_{kl_k};i),
	\]
	\item an action of the symmetric group $\Sigma_k$ on $\mathcal{O}(i_1,\ldots,i_k;i)$,
\end{itemize}
satisfying associativity, unitality and equivariance axioms.

When we restrict to $k=1$, we get a category enriched in $\mathcal{E}$. This category is called the \emph{underlying category} of $\mathcal{O}$ and will be denoted by $\mathcal{O}_1$.

An algebra $\mathcal{A}$ over a coloured operad $\mathcal{O}$ is given by a collection $\mathcal{A}_i$, for $i \in I$, of objects in $\mathcal{E}$, together with maps
\[
	\mathcal{O}(i_1,\ldots,i_k;i) \otimes \mathcal{A}_{i_1} \otimes \ldots \otimes \mathcal{A}_{i_k} \to \mathcal{A}_i,
\]
satisfying associativity, unitality and equivariance axioms.

The restriction of $\mathcal{A}$ to $\mathcal{O}_1$ is called the \emph{underlying functor} of $\mathcal{A}$ and is denoted by $\mathcal{A}^\bullet$.

\subsection{Coends}

We assume that the reader is familiar with the notion of coends. However, since the explicit definition will be used in the proof of Lemma \ref{lemmaformulacoend}, we recall it now. First, recall that a cowedge for a functor $F: C^{op} \times C \to D$ is given by an object $d \in D$ together with maps $\pi_c: F(c,c) \to d$, for all $c \in C$, such that for all morphisms $f: c \to c'$ in $C$, the following square commutes
\begin{equation}\label{cowedge}
\xymatrix{
	F(c',c) \ar[r]^{F(f,id_c)} \ar[d]_{F(id_{c'},f)} & F(c,c) \ar[d]^{\pi_c} \\
	F(c',c') \ar[r]_-{\pi_{c'}} & d
}
\end{equation}
A map between cowedges is a map $t: d \to d'$ such that $t \cdot \pi_c = \pi'_c$ for all $c \in C$. The coend $\int^c F(c,c)$ of $F$ is the initial cowedge.

\subsection{Condensation of a coloured operad}

Let $\mathcal{O}$ be a coloured operad in $\mathrm{Set}$. If $(\mathcal{E},\otimes,e)$ is a cocomplete symmetric monoidal category, there is a monoidal functor $\mathrm{Set} \to \mathcal{E}$ sending a set $X$ to $\coprod_{X} e$, so that $\mathcal{O}$ can be seen as an operad in $\mathcal{E}$. For any colour $i$ of $\mathcal{O}$, we get a functor
\begin{equation}
	\mathcal{O}(-,\ldots,-;i): \mathcal{O}_1^{op} \otimes \ldots \otimes \mathcal{O}_1^{op} \to \mathcal{E}.
\end{equation}
Let us assume that a functor $\delta: \mathcal{O}_1 \to \mathcal{E}$ is fixed.

\begin{definition}\label{definitionrealization}
	For $\mathcal{O}$ a coloured operad and $k \geq 0$, let $\xi(\mathcal{O})_k: \mathcal{O}_1 \to \mathcal{E}$ be the functor defined by the coend formula
	\[
		\xi(\mathcal{O})_k(i) = \int^{i_1,\ldots,i_k} \mathcal{O}(i_1,\ldots,i_k;i) \otimes \delta(i_1) \ldots \otimes \delta(i_k).
	\]
\end{definition}

We assume further that $\mathcal{E}$ is closed, so that the category $\mathcal{E}^{\mathcal{O}_1}$, of functor from $\mathcal{O}_1$ to $\mathcal{E}$, is enriched in $\mathcal{E}$. According to \cite[Proposition 1.8]{BataninBergerLattice}, $\xi(\mathcal{O})$ has the structure of a functor-operad \cite[Definition 4.1]{mccluresmith} on $\mathcal{E}^{\mathcal{O}_1}$.

Let
\[
	\mathrm{Coend}_\mathcal{O}(k) := \underline{\mathrm{Hom}}_{\mathcal{O}_1}(\delta,\xi(\mathcal{O})_k(\delta,\ldots,\delta)),
\]
where $\underline{\mathrm{Hom}}_{\mathcal{O}_1}$ is the internal hom in $\mathcal{E}^{\mathcal{O}_1}$. According to \cite[Proposition 4.4]{mccluresmith}, $\mathrm{Coend}_{\mathcal{O}}$ has a structure of (one-coloured) operad. This one-coloured operad is called \emph{condensation} of $\mathcal{O}$.

\subsection{The complete graph operad}

For $k \geq 0$, let $\mathcal{K}(k) = \mathbb{N}^{k \choose 2} \times \Sigma_k$. For $\sigma \in \Sigma_k$ and $ij \in {k \choose 2}$, let $\sigma_{ij} \in \Sigma_2$ be the permutation of $i$ and $j$ under $\sigma$. We define a partial order on $\mathcal{K}$ as follows. For $(\mu,\sigma),(\nu,\tau) \in \mathcal{K}(k)$, we write $(\mu,\sigma) \leq (\nu,\tau)$ if, for all $ij \in {k \choose 2}$, either $\mu_{ij}<\nu_{ij}$, or $\mu_{ij}=\nu_{ij}$ and $\sigma_{ij}=\tau_{ij}$.

In order to define the operadic structure of $\mathcal{K}$, let us recall that an element of $\mathcal{K}(k)$ can be considered as a complete graph on the vertex set $\{1,\ldots,k\}$ equipped with an edge-labelling by the integers $\mu_{ij}$ and a permutation of the vertices. In this context, the multiplication can be described as insertion of complete graphs inside the vertices of the first graph then renumbering of the vertices, as in the following picture:
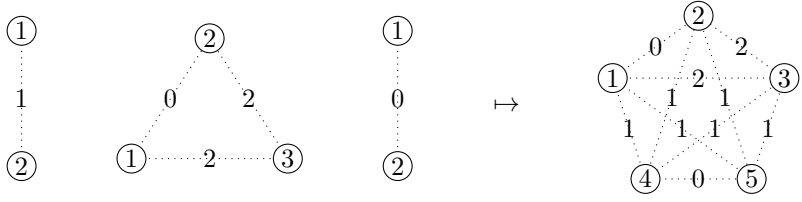
\[
 \begin{tikzpicture}
	 \draw[dotted] (0,-.8) -- (0,1);
	 \draw[fill=white] (0,1) circle (.192) node{$1$};
	 \draw[fill=white] (0,-.8) circle (.192) node{$2$};
	 \draw (0,.1) node{$1$};
	 
	 \begin{scope}[shift={(2.5,-.1)}]
	 \draw[dotted] ({1.2*cos(30)},{-1.2*sin(30)}) -- (0,1) -- ({-1.2*cos(30)},{-1.2*sin(30)}) -- ({1.2*cos(30)},{-1.2*sin(30)});

	 \draw [fill=white] ({-1.2*cos(30)},{-1.2*sin(30)}) circle (.192) node{$1$};
	 \draw [fill=white] (0,1) circle (.192) node{$2$};
	 \draw [fill=white] ({1.2*cos(30)},{-1.2*sin(30)}) circle (.192) node{$3$};
	 
	 \draw (0,{-1.2*sin(30)}) node{$2$};
	 \draw ({.6*cos(30)},{.5-.6*sin(30)}) node{$2$};
	 \draw ({-.6*cos(30)},{.5-.6*sin(30)}) node{$0$};
	 \end{scope}
	 
	 \begin{scope}[shift={(5,0)}]
	 \draw[dotted] (0,-.8) -- (0,1);
	 \draw[fill=white] (0,1) circle (.192) node{$1$};
	 \draw[fill=white] (0,-.8) circle (.192) node{$2$};
	 \draw (0,.1) node{$0$};
	 \end{scope}
	 
	 \draw (6.45,0) node{$\mapsto$};
	 
	 \begin{scope}[shift={(9,0)},scale=1.2]
	 \draw[dotted] ({cos(54)},{-sin(54)}) -- ({cos(18)},{sin(18)}) -- (0,1) -- ({-cos(18)},{sin(18)}) -- ({-cos(54)},{-sin(54)}) -- ({cos(54)},{-sin(54)}) -- (0,1) -- ({-cos(54)},{-sin(54)}) -- ({cos(18)},{sin(18)}) -- ({-cos(18)},{sin(18)}) -- ({cos(54)},{-sin(54)});
	 \draw[fill=white] ({cos(18)},{sin(18)}) circle (.16) node{$3$};
	 \draw[fill=white] (0,1) circle (.16) node{$2$};
	 \draw[fill=white] ({-cos(18)},{sin(18)}) circle (.16) node{$1$};
	 \draw[fill=white] ({-cos(54)},{-sin(54)}) circle (.16) node{$4$};
	 \draw[fill=white] ({cos(54)},{-sin(54)}) circle (.16) node{$5$};
	 
	 \draw (0,{-sin(54)}) node{$0$};
	 \draw ({cos(18)/2+cos(54)/2},{sin(18)/2-sin(54)/2}) node{$1$};
	 \draw ({cos(18)/2},{sin(18)/2+1/2}) node{$2$};
	 \draw ({-cos(18)/2},{sin(18)/2+1/2}) node{$0$};
	 \draw ({-cos(18)/2-cos(54)/2},{sin(18)/2-sin(54)/2}) node{$1$};
	 
	 \draw ({cos(18)/2-cos(54)/2},{sin(18)/2-sin(54)/2}) node{$1$};
	 \draw (0,{sin(18)}) node{$2$};
	 \draw ({cos(54)/2-cos(18)/2},{sin(18)/2-sin(54)/2}) node{$1$};
	 \draw ({cos(54)/2},{1/2-sin(54)/2}) node{$1$};
	 \draw ({-cos(54)/2},{1/2-sin(54)/2}) node {$1$};
	 \end{scope}
 \end{tikzpicture}
\]
Finally, recall that $\mathcal{K}$ admits a filtration. For $m \geq 0$, the $m$-th stage filtration is the suboperad given by $\mathcal{K}_m(k) = \{0,\ldots,m\}^{k \choose 2} \times \Sigma_k$.

\subsection{Multiplicative hyperoperads}\label{subsectionmultiplicativehyperoperads}

A hyperoperad $\mathcal{A}$ in a symmetric monoidal category $(\mathcal{E},\otimes,e)$ is given by a collection of objects in $\mathcal{E}$ over the set of (isomorphism classes of) planar trees together with
\begin{itemize}
	\item for all $n \geq 0$, a map
	\[
	e \to \mathcal{A}(\gamma_n)
	\]
	called \emph{unit}, where $\gamma_n$ is the $n$-th corolla, that is the planar tree with a unique vertex and $n$ leaves,
	\item for all planar trees $S$ and $T$ such that $T$ has as many leaves as the number of edges above a vertex $v$ of $S$, a map
	\[
	\mathcal{A}(S) \otimes \mathcal{A}(T) \to \mathcal{A}(S \circ_v T),
	\]
	where $S \circ_v T$ is the tree obtained by inserting $T$ inside $v$, as in the following picture:
	
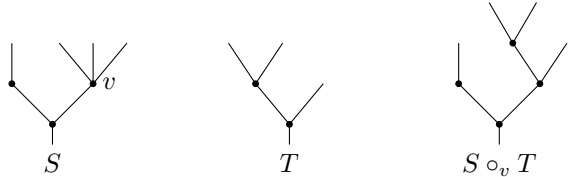
\[
	\begin{tikzpicture}[scale=.9]
	\draw (0,-.3) -- (0,0) -- (-.6,.6) -- (-.6,1.2);
	\draw (0,0) -- (.6,.6) -- (.6,1.2);
	\draw (.1,1.2) -- (.6,.6) -- (1.1,1.2);
	\draw[fill] (0,0) circle (1.2pt);
	\draw[fill] (-.6,.6) circle (1.2pt);
	\draw[fill] (.6,.6) circle (1.2pt) node[right]{$v$};
	\draw (0,-.3) node[below]{$S$};
	
	\begin{scope}[shift={(3.5,0)}]
	\draw (0,-.3) -- (0,0);
	\draw (-.5,.6) -- (0,0) -- (.5,.6);
	\draw (-.9,1.2) -- (-.5,.6) -- (-.1,1.2);
	\draw[fill] (0,0) circle (1.2pt);
	\draw[fill] (-.5,.6) circle (1.2pt);
	\draw (0,-.3) node[below]{$T$};
	\end{scope}
	
	\begin{scope}[shift={(6.6,0)}]
	\draw (0,-.3) -- (0,0) -- (-.6,.6) -- (-.6,1.2);
	\draw (0,0) -- (.6,.6);
	\draw (.2,1.2) -- (.6,.6) -- (1,1.2);
	\draw (-.15,1.8) -- (.2,1.2) -- (.55,1.8);
	\draw[fill] (0,0) circle (1.2pt);
	\draw[fill] (-.6,.6) circle (1.2pt);
	\draw[fill] (.6,.6) circle (1.2pt);
	\draw[fill] (.2,1.2) circle (1.2pt);
	\draw (0,-.3) node[below]{$S \circ_v T$};
	\end{scope}
	\end{tikzpicture}
	\]
\end{itemize}
These maps should satisfy associativity and unitality axioms analogous to the ones in \cite[Definition 11]{markl2008operads}. 
There is an obvious notion of \emph{maps between hyperoperads}, so that hyperoperads form a category. A hyperoperad is called \emph{multiplicative} when it is equipped with a map from the constant hyperoperad $\zeta$ given by the unit $e$ for each planar tree.

\section{Condensation of the operad for multiplicative hyperoperads}\label{sectioncondensation}

%

\subsection{The operad for multiplicative hyperoperads}\label{subsectionoperadformulthyperop}

	In this subsection we will describe explicitly the operad $\mathcal{H}$ for multiplicative hyperoperads. The colours for this operad are (isomorphism classes of) planar trees. Let us define a \emph{circled planar tree} as a planar tree with circles drawn on it, such that the circles never intersect each other and each one of them corresponds unequivocally to a subtree:
	
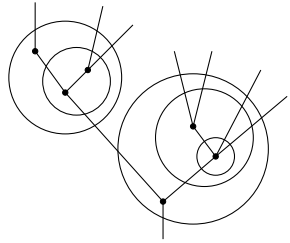
\[
	\begin{tikzpicture}
	\draw (0,-.5) -- (0,0) -- (-1.3,1.45) -- (-1,1.75);
	\draw (-.8,2.6) -- (-1,1.75) -- (-.4,2.35);
	\draw (-1.3,1.45) -- (-1.7,2) -- (-1.7,2.65);
	\draw (0,0) -- (.7,.6) -- (.4,1) -- (.15,2);
	\draw (.4,1) -- (.65,2);
	\draw (1.3,1.75) -- (.7,.6) -- (1.65,1.4);
	
	\draw (.4,.7) circle (1);
	\draw (.55,.85) circle (.65);
	\draw (.7,.6) circle (.25);
	\draw (-1.3,1.65) circle (.75);
	\draw (-1.15,1.6) circle (.45);
	
	\draw[fill] (0,0) circle (1pt);
	\draw[fill] (.7,.6) circle (1pt);
	\draw[fill] (.4,1) circle (1pt);
	\draw[fill] (-1.3,1.45) circle (1pt);
	\draw[fill] (-1.7,2) circle (1pt);
	\draw[fill] (-1,1.75) circle (1pt);
	
%
	\end{tikzpicture}
	\]
	The planar tree \emph{inside} a circle is the subtree corresponding to this circle, but where all the subtrees corresponding to circles further inside have been contracted to a corolla. For example, the planar tree inside the biggest circle of the previous picture is the following:
	
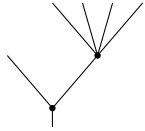
\[
		\begin{tikzpicture}
			\draw (0,-.25) -- (0,0);
			\draw (-.6,.7) -- (0,0) -- (.6,.7);
			\draw (0,1.4) -- (.6,.7) -- (1.2,1.4);
			\draw (.4,1.4) -- (.6,.7) -- (.8,1.4);
			
			\draw[fill] (0,0) circle (1pt);
			\draw[fill] (.6,.7) circle (1pt);
		\end{tikzpicture}
	\]
	For $k \geq 0$ and a $k+1$-tuple $(T_1,\ldots,T_k;T)$ of planar trees, $\mathcal{H}(T_1,\ldots,T_k;T)$ is the set of (isomorphism classes of) circled planar trees, where the circles can be either black or white, such that
	\begin{itemize}
		\item the white circles are numbered from $1$ to $k$ and, for $j \in \{1,\ldots,k\}$, $T_j$ is tree inside the $j$-th white circle,
		\item $T$ is the planar tree obtained by deleting all the circles,
		\item there are no black circles around corollas, no pairs of black circles one directly inside the other, and each black circle must be inside some white circle.
	\end{itemize}
The multiplication of this operad is given by insertion of circled planar trees inside the white circles of the first circled planar tree then deletion of black circles if necessary.

Note that the underlying category of $\mathcal{H}$ is the non-symmetric version $\Omega_p$ of the dendroidal category \cite[Definition 2.2.1]{moerdijktoen}.

\subsection{The complexity map}\label{subsectioncomplexitymap}
Let us assume that $k \geq 0$ and a $k+1$-tuple of planar trees $(T_1,\ldots,T_k;T)$ are given. We define
\[
c: \mathcal{H}(T_1,\ldots,T_k;T) \to \mathcal{K}_3(k)
\]
as the function which sends a circled planar tree to the pair $(\mu,\sigma)$, where for $1 \leq i < j \leq k$,
\[
\mu_{ij} = \begin{cases}
0 &\text{if the circles $i$ and $j$ are one to the left of the other,} \\
1 &\text{if the circles $i$ and $j$ are one below the other,} \\
2 &\text{if the circles $i$ and $j$ are one outside the other.}
\end{cases}
\]
and $\sigma_{ij}$ is neutral (resp. non-neutral) if the circle $i$ (resp. $j$) is to the left, below or outside the circle $j$ (resp. $i$). Note that $c$ is not a map of operads, but we always have the relation
\begin{equation}\label{equationcomplexitymap}
	c(m(o,o_1,\ldots,o_k)) \leq m(c(o),c(o_1),\ldots,c(o_k)).
\end{equation}

\subsection{A formula for the coend}

The objective of this subsection is to prove Lemma \ref{lemmaformulacoend} which describes the coend of Definition \ref{definitionrealization} explicitly.


\begin{definition}
	Let $\mathcal{O}$ be a coloured operad, $i$ a colour of $\mathcal{O}$ and $k \geq 0$. We define $\mathcal{O}_k/i$ as the category whose
	\begin{itemize}
		\item objects are operations $o \in \mathcal{O}(i_1,\ldots,i_k;i)$, where $(i_1,\ldots,i_k)$ is a $k$-tuples of colours of $\mathcal{O}$,
		\item morphisms from $o \in \mathcal{O}(i_1,\ldots,i_k;i)$ to $o' \in \mathcal{O}(i'_1,\ldots,i'_k;i)$ are given by $k$-tuples of unary operations
		\[
			p_1 \in \mathcal{O}(i'_1;i_1),\ldots,p_k \in \mathcal{O}(i'_k;i_k)
		\]
		such that $m(o',p_1,\ldots,p_k) = o$.
	\end{itemize}
\end{definition}

Note that when $k=1$, we get the classical notion of a comma category.

\begin{remark}\label{remarkconceptualexplanation}
	For a planar tree $T$ and $k \geq 0$, the morphisms in $\mathcal{H}_k/T$ replace white circles by bigger white circles and possibly delete black circles. For example, if $T$ is the free living edge and $k=2$, the category contains the following diagram, where the dotted circles represent white circles:
	
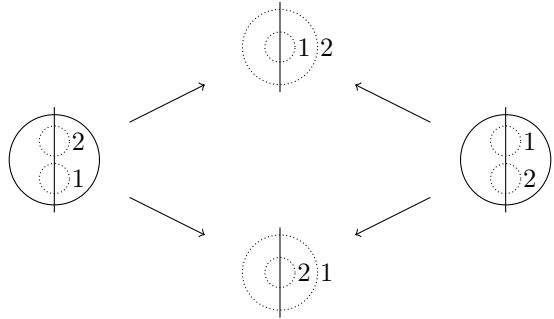
\[
		\begin{tikzpicture}
			\begin{scope}[shift={(0,1.5)}]
			\draw (0,-.6) -- (0,.6);
			\draw[densely dotted] (0,0) circle (.5);
			\draw[densely dotted] (0,0) circle (.2);
			\draw (.32,0) node{$1$};
			\draw (.62,0) node{$2$};
			\end{scope}
	
			\begin{scope}[shift={(0,-1.5)}]
			\draw (0,-.6) -- (0,.6);
			\draw[densely dotted] (0,0) circle (.5);
			\draw[densely dotted] (0,0) circle (.2);
			\draw (.32,0) node{$2$};
			\draw (.62,0) node{$1$};
			\end{scope}
			
			\begin{scope}[shift={(3,0)}]
			\draw (0,-.7) -- (0,.7);
			\draw (0,0) circle (.6);
			\draw[densely dotted] (0,-.25) circle (.2);
			\draw[densely dotted] (0,.25) circle (.2);
			\draw (.32,-.25) node{$2$};
			\draw (.32,.25) node{$1$};
			\end{scope}
	
			\begin{scope}[shift={(-3,0)}]
			\draw (0,-.7) -- (0,.7);
			\draw (0,0) circle (.6);
			\draw[densely dotted] (0,-.25) circle (.2);
			\draw[densely dotted] (0,.25) circle (.2);
			\draw (.32,-.25) node{$1$};
			\draw (.32,.25) node{$2$};
			\end{scope}
			
			\draw[->] (-2,-.5) -- (-1,-1);
			\draw[->] (2,-.5) -- (1,-1);
			\draw[->] (-2,.5) -- (-1,1);
			\draw[->] (2,.5) -- (1,1);
		\end{tikzpicture}
	\]
	There is an analogy between the position of circles on a planar tree, if we assume that these circles can be either one below the other or one outside the other, and the position of distinct points in the plane. This is the conceptual explanation of Theorem \ref{theoremmapfrome2tocoend}.
\end{remark}

In the next lemma, we will assume that $\mathcal{E}=\mathrm{Cat}$ is the symmetric monoidal category of small categories. We will also fix $\delta: \mathcal{O}_1 \to \mathrm{Cat}$ to be the functor which sends $i$ to $\mathcal{O}_1/i$.

\begin{lemma}\label{lemmaformulacoend}
	Let $\mathcal{O}$ be a coloured operad and $i$ a colour of $\mathcal{O}$. For all $k \geq 0$, the coend $\xi(\mathcal{O})_k(i)$ of Definition \ref{definitionrealization} is isomorphic to $\mathcal{O}_k/i$.
\end{lemma}

\begin{proof}
	We will prove that $\mathcal{O}_k/i$ is the desired coend. First we prove that it is a cowedge. Let $i_1,\ldots,i_k$ be a $k$-tuple of colours of $\mathcal{O}$. There is an obvious functor
	\[
		\pi_{i_1,\ldots,i_k}: \mathcal{O}(i_1,\ldots,i_k;i) \times \delta(i_1) \times \ldots \times \delta(i_k) \to \mathcal{O}_k/i
	\]
	given by the multiplication of $\mathcal{O}$. The squares such as \ref{cowedge} commute thanks to the associativity of $\mathcal{O}$.
	
	Now we prove that it is the initial cowedge. Let $X \in \mathcal{E}$ equipped with functors
	\[
		\psi_{i_1,\ldots,i_k}: \mathcal{O}(i_1,\ldots,i_k;i) \times \delta(i_1) \times \ldots \times \delta(i_k) \to X
	\]
	If $t: \mathcal{O}_k/i \to X$ is a map between cowedges, then for $o \in \mathcal{O}(i_1,\ldots,i_k;i)$,
	\[
		t(o) = t \cdot \pi_{i_1,\ldots,i_k} (o,id_{i_1},\ldots,id_{i_k}) = \psi_{i_1,\ldots,i_k} (o,id_{i_1},\ldots,id_{i_k}),
	\]
	which proves that it is unique. To prove that it exists, we need to check that this indeed gives a map between cowedges. Let $(f_1,\ldots,f_k): (i_1,\ldots,i_k) \to (i'_1,\ldots,i'_k)$ be a morphism in $(\mathcal{O}_1)^k$. Then we have
	\begin{align*}
		t \cdot \pi_{i'_1,\ldots,i'_k}(o,f_1,\ldots,f_k) &= t \cdot m (o,f_1,\ldots,f_k)\\
		&= \psi_{i_1,\ldots,i_k} (m(o,f_1,\ldots,f_k),id_{i_1},\ldots,id_{i_k})\\
		&= \psi_{i'_1,\ldots,i'_k} (o,f_1,\ldots,f_k).
	\end{align*}
	The last equality comes from the fact that $X$ is a cowedge, so the squares such as \ref{cowedge} commute. This concludes the proof.
\end{proof}

%
%
%
%
%
%
%
%
%

\subsection{Proof of the main result}

For $(\mu,\sigma) \in \mathcal{K}_3(k)$ and a planar tree $T$, let $\mathcal{H}_{(\mu,\sigma)}/T$ be the full subcategory of $\mathcal{H}_k/T$ of operations $o \in \mathcal{H}(T_1,\ldots,T_k;T)$ such that $c(o) \leq (\mu,\sigma)$. We have the following lemma, which is inspired from \cite[Lemma 14.8]{mccluresmith}.

\begin{lemma}\label{lemmacontractible}
	Let $(\mu,\sigma) \in \mathcal{K}_3(k)$ and $T$ be a planar tree. If for all $ij \in {k \choose 2}$, $\mu_{ij} \geq 1$, then the nerve of $\mathcal{H}_{(\mu,\sigma)}/T$ is contractible.
\end{lemma}

\begin{proof}
	We proceed by induction on $k$. When $k=0$, $\mathcal{H}_{(\mu,\sigma)}/T$ contains only one element which is given by the planar tree $T$ without circles. So the result is trivial. Now, let us assume that $k \geq 1$. Let $(\mu',\sigma')$ be the complete graph obtained by deleting the vertex $i := \sigma^{-1}(1)$ from $(\mu,\sigma)$. There is a functor $F: \mathcal{H}_{(\mu,\sigma)}/T \to \mathcal{H}_{(\mu',\sigma')}/T$ which deletes the white circle $i$ and some black circles if necessary. Let us prove that $F$ induces a weak equivalence between nerves. According to Quillen's Theorem A, it suffices to show that for all $o' \in \mathcal{H}_{(\mu',\sigma')}/T$, the nerve of $F/o'$ is contractible. As in \cite[Section 5.3]{cisinski}, we consider the canonical inclusion functor $F_{o'} \to F/o'$, where $F_{o'}$ is the fibre of $F$ over $o'$. First, note that $F_{o'}$ has a terminal object. Indeed, since $i = \sigma^{-1}(1)$ and for all $ij \in {k \choose 2}$, $\mu_{ij} \geq 1$, the objects of $F_{o'}$ consists of the circled planar $o'$ where a white circle has been added, which is either below, above or outside all other white circles. The terminal object is when this added white circle is maximal. Therefore $F_{o'}$ has a contractible nerve. Finally, note that $F_{o'}$ is a reflective subcategory of $F/o'$. Indeed, the inclusion functor has a left adjoint which sends $F(o) \to o'$ to the circled planar tree obtained by adding the white circle $i$ of $o$ in $o'$. This concludes the proof.
%
\end{proof}

We now work in the case where $\mathcal{E}=\mathrm{Top}$ is the category of compactly generated spaces. The functor $\delta: \Omega_p \to \mathrm{Top}$ sends $T$ to $B(\Omega_p/T)$, where $B$ is the \emph{classifying space} functor, given by geometric realization of the nerve.

\begin{theorem}\label{theoremmapfrome2tocoend}
	There is a map of operads from an $E_2$-operad to $\mathrm{Coend}_\mathcal{H}$.
\end{theorem}

\begin{proof}
	The proof is inspired from the proof of \cite[Theorem 3.8]{BataninBergerLattice}. Let $\iota: \mathcal{K}_2 \to \mathcal{K}_3$ be the map of operads which sends $\mu_{ij}$ to $\mu_{ij}+1$. Let $\hat{\mathcal{H}}$ be the coloured operad whose set of colours is the set of planar trees. For $k \geq 0$ and $(T_1,\ldots,T_k;T)$ a $k+1$-tuple of planar trees, $\hat{\mathcal{H}}(T_1,\ldots,T_k;T)$ is the set of pairs $(o,(\mu,\sigma))$, such that $o \in \mathcal{H}(T_1,\ldots,T_k;T)$, $(\mu,\sigma) \in \mathcal{K}_2(k)$ and $c(o) \leq \iota(\mu,\sigma)$. Since $c$ satisfies \eqref{equationcomplexitymap} and $\iota$ is a map of operad, $\hat{\mathcal{H}}$ is indeed an operad.
	
	We will prove that there is a zigzag of operadic maps
	\begin{equation}\label{equationzigzag}
	\mathrm{Coend}_\mathcal{H} \leftarrow \mathrm{Coend}_{\hat{\mathcal{H}}} \xrightarrow{\sim} B(\mathcal{K}_2)
	\end{equation}
	where the right map is a weak equivalence . First, there is an obvious projection map $\hat{\mathcal{H}} \to \mathcal{H}$ which forgets the pair $(\mu,\sigma)$. The left map in \eqref{equationzigzag} is induced by this map. According to Lemma \ref{lemmaformulacoend}, we have
	\[
		\mathrm{Coend}_{\hat{\mathcal{H}}}(k) = \holim_{T \in \Omega_p} B(\hat{\mathcal{H}}_k/T).
	\]
	The category $\hat{\mathcal{H}}_k/T$ can be obtained as the Grothendieck construction for the functor $\Phi: \mathcal{K}_2(k) \to \mathrm{Cat}$ sending $(\mu,\sigma)$ to $\mathcal{H}_{\iota(\mu,\sigma)}/T$. Therefore, according to Thomason's Theorem, the classifying space of $\hat{\mathcal{H}}_k/T$ can be computed as the homotopy colimit of $B\Phi$. The right map in \eqref{equationzigzag} is induced by the unique natural transformation from $\Phi$ to the constant functor to the terminal category. This natural transformation induces a pointwise weak equivalence between classifying spaces according to Lemma \ref{lemmacontractible}. Therefore, the right map in \eqref{equationzigzag} is a weak equivalence of operads. 
%
	The conclusion follows from \cite[Theorem 1.16]{bergercombinatorial}, which states that $B(\mathcal{K}_2)$ is an $E_2$-operad.
\end{proof}

\begin{remark}
	One could be tempted to define $\hat{\mathcal{H}}$ as the coloured operad whose operations are pairs operations $o \in \mathcal{H}$ and $(\mu,\sigma) \in \mathcal{K}_3$ such that $c(o) \leq (\mu,\sigma)$. Then one could try to construct a weak equivalence of operads $\mathrm{Coend}_{\hat{\mathcal{H}}} \to B(\mathcal{K}_3)$. Unfortunately this does not work. The reason is that Lemma \ref{lemmacontractible} does not apply for $(\mu,\sigma)$ such that there is $ij \in {k \choose 2}$ where $\mu_{ij}=0$. For example, if $T$ is a linear tree, $\mathcal{H}_{(\mu,\sigma)}/T$ is empty. Indeed, it is impossible to draw two circles, one to the left of the other, on a linear tree.
\end{remark}

\begin{corollary}
	For any multiplicative hyperoperad $\mathcal{A}$, there is an $E_2$-action on the homotopy limit of the underlying functor of $\mathcal{A}$.
\end{corollary}

\begin{proof}
	According to \cite[Proposition 1.5]{BataninBergerLattice}, since $\mathcal{A}$ is an algebra of $\mathcal{H}$, there is a $\mathrm{Coend}_\mathcal{H}$-action on $\holim_{\Omega_p} \mathcal{A}^\bullet$. It can be restricted to an $E_2$-action according to Theorem \ref{theoremmapfrome2tocoend}.
\end{proof}

%
%
%
%
%
%
%

\bibliographystyle{plain}
\bibliography{condensation}

\end{document}